\documentclass[textwidth=125mm, textheight=195mm]{article}

\usepackage{amssymb,amsmath,amscd,amsthm}

\date{}

\newcommand{\Z}{{\mathbb Z}}
\newcommand{\R}{{\mathbb R}}
\newcommand{\C}{{\mathbb C}}

\newcommand{\D}{{\mathbb D}}

\newcommand{\E}{{\mathcal E}}

\newcommand{\atr}{{\mathrm{atr}}}
\newcommand{\Int}{{\mathrm{int}}}

\newcommand{\SL}{{\mathrm{SL}}}
\newcommand{\SO}{{\mathrm{SO}}}

\newcommand{\Hd}{{\mathrm{H}}}
\newtheorem{theorem}{Theorem}[section]
\newtheorem{lemma}[theorem]{Lemma}
\newtheorem{prop}[theorem]{Proposition}
\newtheorem{coro}[theorem]{Corollary}

\theoremstyle{definition}

\theoremstyle{definition}
\newtheorem*{defi}{Definition}

\newcommand{\tr}{\mathrm{tr} }

\begin{document}



\title{Spectral Homogeneity of Discrete One-Dimensional Limit-Periodic Operators}

\author{Jake Fillman\thanks{Corresponding author}  \thanks{J.\ F.\ was supported in part by NSF grant DMS--1067988.}}

\maketitle

\begin{center}
Department of Mathematics, Rice University

6100 Main Street, MS-136

Houston, TX 77005, USA

jdf3@rice.edu
\end{center}

\begin{abstract}
We prove that a dense subset of limit periodic operators have spectra which are homogeneous Cantor sets in the sense of Carleson. Moreover, by using work of Egorova, our examples have purely absolutely continuous spectrum. The construction is robust enough to extend the results to arbitrary $p$-adic hulls by using the dynamical formalism proposed by Avila. The approach uses Floquet theory to break up the spectra of periodic approximants in a carefully controlled manner to produce Cantor spectrum and to establish the lower bounds needed to prove homogeneity.
\end{abstract}

\section{Introduction}

We are interested in spectral characteristics of self-adjoint operators on $\ell^2(\Z)$ of the form
\begin{equation} \label{eq:so}
(H_V \psi)(n)
=
\psi(n-1) + \psi(n+1) + V(n) \psi(n), \, n \in \Z,
\end{equation}
where $V \in \ell^\infty(\Z)$ plays the role of an electrostatic potential. In particular, operators of the form \eqref{eq:so} model one-dimensional one-body tight binding Hamiltonians, and thus, they provide a rich class of toy models in quantum mechanics and spectral theory. If $V$ is a periodic sequence, then the spectral theory of $H_V$ is quite well-understood by way of Floquet-Bloch theory. Indeed, any spectral-theoretic object (the spectrum, the density of states, the spectral measures, etc.) can be described quite explicitly; an inspired reference for this subject is \cite[Chapter~5]{simszego}. As soon as $V$ departs from the class of periodic potentials, the spectral characteristics of $H_V$ become significantly more subtle and elusive.

In this paper, we focus on the class of \emph{limit-periodic} operators, that is, operators of the form \eqref{eq:so} for which  the potential can be written as an $\ell^\infty$-limit of periodic sequences; see \cite{avila, damanikgan1, damanikgan2, damanikgan3, gan2010, gankrueger}. A typical example of such a potential is furnished by
$$
V(n)
=
\sum_{j=1}^\infty 2^{-j^2} \cos\left( \frac{2\pi n}{j!}\right).
$$
More specifically, we are concerned with homogeneity of the spectra of limit-periodic operators. Loosely speaking, a homogeneous closed subset of $\R$ is one which has a uniform positive density in arbitrarily small neighborhoods of each of its points. The precise definition follows (compare \cite{carleson83}). 

\begin{defi}
We say that a closed set $K \subseteq \R$ is \emph{homogeneous} (in the sense of Carleson) if there exist $\tau, \, \delta_0 > 0$ such that
\begin{equation} \label{eq:homog:def}
|B_\delta(x) \cap K |
\geq
\tau \delta
\quad
\text{for every } 0 < \delta \leq \delta_0 \text{ and } x \in K,
\end{equation}
where $B_\delta(x) = (x - \delta, x + \delta)$ denotes the  $\delta$-neighborhood of $x$. If we want to emphasize the relative density of $K$, we will say that a compact set which satisfies \eqref{eq:homog:def} for some $\delta_0 > 0$ is $\tau$-\emph{homogeneous}.
\end{defi}

Homogeneity of closed subsets of $\R$ is important from the point of view of inverse spectral theory. In particular, if $K$ is a homogeneous compact set, then the space of Jacobi matrices which have spectrum $K$ and are reflectionless thereupon is known to consist of almost-periodic operators by a theorem of Sodin and Yuditskii \cite{SY97}; moreover, Poltoratski and Remling have proved that the spectral measures of such Jacobi matrices will be purely absolutely continuous \cite{poltrem09}. There are analogous results for the inverse spectral theory of continuum Schr\"odinger operators and CMV matrices in \cite{GY99, SY95} and \cite{GZ09}, respectively.
\newline

Generically, the spectra of limit-periodic operators are of zero Lebesgue measure and so cannot be homogeneous in this sense \cite[Corollary~1.2]{avila}. On the other hand, the spectrum corresponding to any periodic potential will be a finite union of closed, bounded intervals; such a set is clearly $1$-homogeneous. In order to examine the interplay between inverse and direct spectral perspectives, it is of interest to apply direct spectral methods to construct almost-periodic examples with more exotic spectra which are nonetheless homogeneous in the sense of Carleson and which have purely absolutely continuous spectrum. This goal has been pursued in the setting of continuum quasi-periodic potentials in the regime of small coupling  \cite{DGL}. We can accomplish this in the class of limit-periodic operators because they are approximated by periodic operators in the operator norm topology. It turns out that a careful perturbative argument proves that the set of potentials with homogeneous Cantor spectrum is dense in the space of limit-periodic potentials. Moreover, by using work of Egorova, we are able to control the spectral type and produce purely absolutely continuous spectrum \cite{egorova}.

In fact, we will prove a more general result. Since spectral homogeneity is of interest from the point of view of inverse spectral theory, the natural family of tri-diagonal operators with which one should work is that of \emph{Jacobi operators}, i.e., operators of the form $J=J_{a,b}:\ell^2(\Z) \to \ell^2(\Z)$, defined by
\begin{equation}\label{eq:jo:def}
(J\psi)(n)
=
a(n-1) \psi(n-1) + a(n) \psi(n+1) + b(n) \psi(n),
\, n \in \Z,
\end{equation}
where $a$ and $b$ are bounded sequences of real numbers; see \cite{teschljacobi}. We will also always assume that $a$ is positive and bounded away from zero. In this context, our main theorem takes the following form.

\begin{theorem} \label{t:jacobi:homog}
Fix a periodic sequence $a > 0$, let $\mathcal L$ denote the set of real-valued limit-periodic sequences, and denote by $\mathcal H_\tau^a$ the set of $b \in \mathcal L$ so that $\sigma(J_{a,b})$ is a $\tau$-homogeneous Cantor set and such that the spectrum of $J_{a,b}$ is purely absolutely continuous. Then $\mathcal H_\tau^a$ is dense in $\mathcal L$ for every $\tau < 1$.
\end{theorem}

\noindent \textit{Remark.} By a Cantor set, we mean a totally disconnected compact set with no isolated points. In particular, each element of $\mathcal H_\tau^a$ is obviously aperiodic. Moreover, a Cantor set clearly cannot be 1-homogeneous, so Theorem~\ref{t:jacobi:homog} is optimal.
\newline

As an immediate corollary, we see that the set of limit-periodic Jacobi parameters which produce purely absolutely continuous spectrum supported on a homogeneous Cantor set is dense in the natural space of Jacobi parameters. More precisely, define
$$
\mathcal J_C
=
\left\{ (a,b) : C^{-1} \leq a(n) \leq C, \, -C \leq b(n) \leq C \, \text{ for all } n \in \Z \right\}
$$
for each $C > 0$, and endow $\mathcal J_C$ with the relative topology that it inherits as a subspace of $\ell^\infty(\Z) \times \ell^\infty(\Z)$. Let $\mathcal L_C \subseteq \mathcal J_C $ denote the set of Jacobi parameters which are limit-periodic (i.e.\ $a$ and $b$ are both limit-periodic sequences). One then has the following corollary of Theorem~\ref{t:jacobi:homog}.

\begin{coro} \label{c:jacobi:homdense}
For each $\tau < 1$, $\mathcal H_{\tau,C}$ is dense in $\mathcal L_C$ with respect to the $\ell^\infty$ topology, where $\mathcal H_{\tau,C}$ denotes the set of $(a,b) \in \mathcal L_{C}$ for which $\sigma(J_{a,b})$ is a $\tau$-homogeneous Cantor set  and $J_{a,b}$ has purely absolutely continuous spectrum.
\end{coro}

By taking $a \equiv 1$ in Theorem~\ref{t:jacobi:homog}, we obtain the first claimed result -- the set of limit-periodic Schr\"odinger operators with purely absolutely continuous spectrum supported on a Carleson-homogeneous Cantor set is $\ell^\infty$-dense in the space of all limit-periodic potentials.

\begin{coro}\label{c:lpspec:homog}
For each $\tau < 1$, let $\mathcal H_\tau^{\mathrm S} \subseteq \mathcal L$ be the set of $V \in \mathcal L$ such that $\sigma(H_V)$ is a $\tau$-homogeneous Cantor set and $H_V$ has purely absolutely continuous spectrum. Then $\mathcal H_\tau^{\mathrm S} $ is $\ell^\infty$-dense in $\mathcal L$ for every $\tau < 1$.
\end{coro}

If one considers inverse spectral theory of unitary operators on the circle rather than the inverse spectral theory of self-adjoint operators on the real line, one is naturally led to the class of CMV operators. Specifically, given a sequence $\alpha$ of complex numbers such that $\alpha(n) \in \D = \{ z \in \C : |z| < 1\}$ for every $n \in \Z$, the associated CMV operator $\E = \E_\alpha$ is defined by the matrix representation
$$
\E_\alpha
=
\begin{pmatrix}
\ddots & \ddots & \ddots &&&&&  \\
 & a(0) & b(1) & d(1) &&& & \\
& c(0) & a(1) & c(1) &&& & \\
&  & b(2) & a(2) & b(3) & d(3) & & \\
& & d(2) & c(2) & a(3) & c(3) &  &  \\
& &&& b(4) & a(4) & b(5) & \\
& &&& d(4) & c(4) & a(5) &   \\
& &&&& \ddots & \ddots &  \ddots
\end{pmatrix}
$$
with respect to the standard basis of $\ell^2(\Z)$, where
\begin{align*}
\rho(n) & = \sqrt{1-|\alpha(n)|^2} \\
a(n) & = -\overline{\alpha(n)} \alpha(n-1)\\
b(n) & = \overline{\alpha(n)} \rho(n-1)\\
c(n) & = -\rho(n) \alpha(n-1)\\
d(n) & = \rho(n) \rho(n-1).
\end{align*}
See \cite{simopuc1, simopuc2} for more detailed information. By straightforward modifications to the proof of Theorem~\ref{t:jacobi:homog}, one obtains the following analog in the realm of CMV operators. Notice that we do not claim to produce purely absolutely continuous spectrum in this setting. It is likely true that our construction gives purely absolutely continuous spectrum in the CMV setting, but the paper of Egorova on which we rely to control the spectral type focuses on the Jacobi case -- we will address this gap in a forthcoming article \cite{Fpreprint}.

\begin{theorem} \label{t:homog:cmv}
Let $\mathcal L_{\D}$ denote the set of limit-periodic complex-valued sequences $\alpha$ with $\alpha(n) \in \D$ for every $n \in \Z$, and, for each $\tau < 1$, denote by $\mathcal H_{\tau}^{\mathrm{CMV}}$ the set of $\alpha \in \mathcal L_{\D}$ such the that $\sigma(\E_\alpha)$ is a $\tau$-homogeneous Cantor set.\footnote{We call a closed subset $K \subseteq \partial \D$ $\tau $-\emph{homogeneous} if and only if it satisfies a bound of the form \eqref{eq:homog:def} with $|\cdot|$ interpreted as arc-length measure on $\partial \D$.} Then $\mathcal H_{\tau}^{\mathrm{CMV}}$ is dense in $\mathcal L_{\D}$ for every $\tau < 1$.
\end{theorem}

It is frequently profitable to imbed limit-periodic sequences into a dynamical context. Specifically, any limit-periodic sequence is Bohr almost-periodic, and so its hull naturally enjoys the structure of a compact abelian topological group. Moreover, it is well-known that an almost-periodic sequence is limit-periodic if and only if its hull is totally disconnected; a detailed discussion of this may be found in \cite[Section~2]{avila}. In light of this, the following definition is natural.

\begin{defi} A \emph{Cantor group} is a compact, abelian, totally disconnected topological group. A \emph{monothetic group} is a topological group  which contains a dense cyclic subgroup. A generator of this dense subgroup is referred to as a \emph{topological generator} of the monothetic group.
\end{defi} 

Standard examples of monothetic Cantor groups include the additive group of $p$-adic integers and the profinite completion of $\Z$. More generally, the class of Cantor groups precisely coincides with the class of infinite profinite abelian groups; see \cite{ribezal,wilson}, for example.

As a consequence of this characterization of limit-periodic sequences via their hulls, it follows that limit-periodic sequences are precisely those which can be generated by continuously sampling along orbits of a minimal translation of a monothetic Cantor group; compare \cite[Lemma~2.2]{avila}. More precisely, a complex-valued sequence $s$ is limit-periodic if and only if one can produce a monothetic Cantor group $\Omega$, a topological generator $\theta$ of $\Omega$, an element $\omega \in \Omega$, and $f \in C(\Omega,\C)$ such that
\begin{equation} \label{eq:lp:cantdef}
s(n)
=
s_\omega^f(n)
:=
f(n\theta + \omega),
\, n \in \Z.
\end{equation}
Given $f \in C(\Omega,\R)$ and a $p$-periodic positive sequence $a$, one obtains Jacobi operators $J_{a,\omega}^f$ with Jacobi parameters $(a,b_\omega^f)$, where $b_\omega^f = s_\omega^f$, as in \eqref{eq:lp:cantdef}. By a standard argument using minimality and strong operator convergence, there exists a deterministic compact set $\Sigma_a^f \subseteq \R$ with $\Sigma_a^f = \sigma(J_{a,\omega}^f)$ for every $\omega \in \Omega$. 

Similarly, if we take $g \in C(\Omega,\D)$, then, for each $\omega \in \Omega$, we obtain a limit-periodic CMV operator $\E_\omega^g$ defined by $\alpha_\omega^g = s_\omega^g$ with $s$ defined by \eqref{eq:lp:cantdef}. As in the Jacobi case, there is a fixed compact set $\Sigma^g \subseteq \partial \D$ with $\sigma(\E_\omega^g) = \Sigma^g$ for every $\omega \in \Omega$.

This point of view is particularly pleasant, since one may fix the underlying dynamics (i.e.\ $\Omega$ and $\theta$) and consider the dependence of spectral properties on $f , g \in C(\Omega)$. Our proofs of homogeneity are robust enough to pass to this setting and produce a dense set of elements of $C(\Omega)$ which produce $\tau$-homogeneous Cantor spectrum.

\begin{theorem} \label{t:jacobi:cantgroup:homog}
Fix a monothetic Cantor group $\Omega$, a topological generator $\theta \in \Omega$, a positive periodic sequence $a$, and $\tau < 1$. Then there is a dense subset $\mathcal H_\tau^a \subseteq C(\Omega,\R)$ such that $\Sigma_a^f$ is a $\tau$-homogeneous Cantor set and $J_{a,\omega}^f$ has purely absolutely continuous spectrum for every $f \in \mathcal H_\tau^a$ and every $\omega \in \Omega$.
\end{theorem}

\begin{theorem} \label{t:homog:cmv:cantgroup}
Fix a monothetic Cantor group $\Omega$, a topological generator $\theta$, and denote by $\mathcal H_\tau^{\mathrm{CMV}} \subseteq C(\Omega,\D)$ the set of $g$ such that $\Sigma^g$ is a $\tau$-homogeneous Cantor set. For each $\tau < 1$, $\mathcal H_\tau^{\mathrm{CMV}}$ is dense in $C(\Omega,\D)$.
\end{theorem}

The structure of the paper is as follows. In Section~\ref{sec:lemmas}, we recall a few standard facts from functional analysis and some necessary pieces of Floquet theory and use these ingredients to prove a gap-opening lemma. This lemma is then used in Section~\ref{sec:proof} to prove Theorems~\ref{t:jacobi:homog} and \ref{t:jacobi:cantgroup:homog}. Clearly, these theorems imply Corollaries~\ref{c:jacobi:homdense} and \ref{c:lpspec:homog}. Finally, Section~\ref{sec:cmvproofs} discusses the necessary modifications to the proofs in the CMV case to obtain Theorems~\ref{t:homog:cmv} and \ref{t:homog:cmv:cantgroup}. The appendix proves a version of a band length estimate for periodic Jacobi matrices which is due to Deift-Simon and Avila in the Schr\"odinger case.

\section*{Acknowledgements}

The author is grateful to David Damanik for comments which improved the exposition, and to Artur Avila for a helpful discussion on the band length estimate from Appendix~A. He is grateful to Milivoje Lukic and Darren Ong for helpful conversations regarding the extension of the main theorems to the CMV setting, and to Fritz Gesztesy for helpful comments on the literature. The author would also like to thank the anonymous referee for suggesting the more elegant proof of Lemma~\ref{l:breakpts} which appears in this version.

\section{Preliminaries} \label{sec:lemmas}

\subsection{The Hausdorff Metric} 

For our proof of Theorem~\ref{t:jacobi:homog}, we will make use of two facts about the Hausdorff metric, whose definition we briefly recall. Given two compact subsets $F,K \subseteq \R$, put
\begin{equation} \label{eq:hdmetric:def}
d_{\Hd}(F,K)
:=
\inf\{ \varepsilon > 0 : F \subseteq B_\varepsilon(K)  \text{ and } K \subseteq B_\varepsilon(F) \},
\end{equation}
where $B_\varepsilon(X)$ denotes the open $\varepsilon$-neighborhood of the set $X \subseteq \R$. The function $d_{\Hd}$ defines a metric on the space of compact subsets of $\R$, known as the \emph{Hausdorff metric}. The following propositions are standard. Since the proofs are short, we include them for the convenience of the reader.

\begin{prop} \label{p:lebmsr:semicont}
Suppose that $(F_n)_{n=1}^\infty$ and $(K_n)_{n=1}^\infty$ are sequences of compact subsets of $\R$. If there exist compact sets $F$ and $K$ such that $F_n \to F$ and $K_n \to K$ with respect to $d_{\Hd}$ as $n \to \infty$, then
$$
|F \cap K|
\geq
\limsup_{n\to\infty}|F_n \cap K_n|.
$$
\end{prop}

\begin{proof}
Given $\varepsilon > 0$, we may use compactness of $F \cap K$ to choose finitely many open intervals $I_1,\ldots, I_m$ with $F \cap K \subseteq O := \bigcup_{j=1}^m I_j$ and $|O| <  |F \cap K| + \varepsilon/2$. Now, take
$$
\delta = \frac{\varepsilon}{4m}.
$$
It is easy to see that $F_n \cap K_n \subseteq B_\delta(F \cap K)$ for all sufficiently large $n$. For such large $n$, one then has $F_n \cap K_n \subseteq B_\delta(O)$, which yields
$$
|F_n \cap K_n|
\leq
|B_\delta(O)|
\leq
|O| + 2m \delta
<
|F \cap K| + \varepsilon.
$$
This argument clearly implies the desired semicontinuity statement.
\end{proof}

\begin{prop} \label{p:specdist}
If $S$ and $T$ are bounded self-adjoint operators on a Hilbert space $\mathcal H$, then
\begin{equation} \label{eq:specdist}
d_{\Hd}(\sigma(S),\sigma(T))
\leq
\|S - T \|.
\end{equation}
\end{prop}

\begin{proof}
Let $\delta = \|T-S\|$, and suppose $x \in \R$ satisfies $d(x,\sigma(T)) > \delta$. In particular, $T-x$ is invertible and, by the spectral theorem, one has 
$$
\left\|(T-x)^{-1} \right\|^{-1} = d(x,\sigma(T)) > \delta.
$$
By an easy geometric series argument, it follows that 
$$
S-x = (T-x) + (S-T)
=
(T-x)\left( I + (T-x)^{-1} (S-T)\right)
$$
is invertible, i.e., $x \notin \sigma(S)$. Thus, the $\delta$-neighborhood of $\sigma(T)$ contains $\sigma(S)$. By symmetry, one may run the previous argument with the roles of $S$ and $T$ reversed, which suffices to establish \eqref{eq:specdist}.
\end{proof}

\subsection{Floquet Theory}

In order to describe our main gap-opening lemma, we give a very brief overview of the necessary highlights of Floquet theory for periodic Jacobi operators and prove a minor variant of a gap-opening lemma due to Avila. Suppose $a,b\in \ell^\infty(\Z)$ are $p$-periodic for some $p \in \Z_+$, that is,
$$
a(n+p) = a(n), \, \, b(n+p) = b(n)
\quad
\text{for all } n \in\Z.
$$
Given $E \in \R$, the study of the eigenvalue equation
\begin{equation} \label{eq:eveq}
a(n-1) u(n-1) + a(n) u(n+1) + b(n) u(n) = E u(n)
\quad
\text{for all }n \in \Z
\end{equation}
leads one to define the transfer matrices $T_E = T_E^{(a,b)}$ and $A_E = A_E^{(a,b)}$ via
$$
T_E(n)
=
\frac{1}{a(n)}
\begin{pmatrix} E - b(n) & -1 \\ a(n)^2 & 0 \end{pmatrix},
\quad
A_E(n) 
=
\begin{cases}
T_E(n) \cdots T_E(1) & n \geq 1 \\
I & n = 0 \\
T_E(n+1)^{-1} \cdots T_E(0)^{-1} & n \leq -1
\end{cases}
$$
Specifically, a complex-valued sequence $u$ satisfies \eqref{eq:eveq} if and only if
$$
\begin{pmatrix} u(n+1) \\ a(n) u(n) \end{pmatrix}
=
A_E(n)
\begin{pmatrix} u(1) \\ a(0) u(0) \end{pmatrix}
\text{ for every } n \in \Z.
$$
The \emph{monodromy matrix} of $J_{a,b}$ is the transfer matrix over a full period; more precisely,
$$
\Phi_E
=
A_E^{(a,b)}(p)
=
T_E(p) \cdots T_E(1).
$$
The \emph{discriminant} of $J_{a,b}$ is defined by $D(E) = \tr(\Phi_E)$. 

One can also consider restrictions of $J_{a,b}$ with suitable periodic or antiperiodic boundary conditions. Specifically, let
\begin{equation} \label{eq:bprob}
J_{a,b}^{p,\pm}
=
\begin{pmatrix}
b(1) & a(1) &   &  & \pm a(p) \\
a(1) & b(2) & a(2)  &  &  \\
 & \ddots & \ddots & \ddots  & & \\
& & a(p-2) & b(p-1) & a(p-1) \\
\pm a(p) && & a(p-1) & b(p)
\end{pmatrix}.
\end{equation}
It is easy to see that $E$ is an eigenvalue of $J_{a,b}^{p,+}$ if and only if there is a nontrivial $p$-periodic solution $u$ of \eqref{eq:eveq} and $E$ is an eigenvalue of $J_{a,b}^{p,-}$ if and only if there is a nontrivial $p$-antiperiodic solution of \eqref{eq:eveq}. Specifically, if $u$ is a nontrivial eigenvector of $J_{a,b}^{p,\pm}$ with eigenvalue $E$, then $u$ can be extended to a two-sided sequence on $\Z$ such that \eqref{eq:eveq} holds and $u(n+p) = \pm u(n)$ for all $n \in \Z$.

It is well-known that the spectrum of $J_{a,b}$ can be determined either from the polynomial $D$ or from the matrices $J_{a,b}^{p,\pm}$. We summarize the relevant facts in the following theorem. Proofs and further details can be found in \cite[Chapter~5]{simszego}.

\begin{theorem} \label{t:floq}
If $a \in (-2,2)$, then all solutions of the equation $D(z) = a$ are real and simple. If $a = \pm 2$, then all solutions of $D(z) = a$ are real and of multiplicity at most two. A solution $E$ of $D(E) = \pm 2$ is of multiplicity two if and only if $\Phi_E = \pm I$. If $\alpha_j$ and $\beta_j$ denote the solutions of $D = \pm 2$ {\rm(}with multiplicity{\rm)}, ordered so that
$$
\alpha_1 \leq \beta_1 \leq \alpha_2 \leq \beta_2 \leq \cdots \leq \alpha_{p-1} \leq \beta_{p-1} \leq \alpha_p \leq \beta_p,
$$
then $\alpha_j < \beta_j$ for each $1 \leq j \leq p$, and
$$
\sigma(J_{a,b})
=
\bigcup_{j=1}^p [\alpha_j,\beta_j]
=
\{ E \in \R : |D(E)| \leq 2 \}.
$$
Moreover, the $\alpha$'s and $\beta$'s comprise the set of all eigenvalues of $J_{a,b}^{p,\pm}$. Specifically, the eigenvalues of $J_{a,b}^{p,+}$ are $\beta_p, \alpha_{p-1}, \beta_{p-2}, \alpha_{p-3},\ldots$, while the eigenvalues of $J_{a,b}^{p,-}$ are $\alpha_p, \beta_{p-1}, \alpha_{p-2}, \beta_{p-3},\ldots$.
\end{theorem}

\begin{defi}
We call intervals of the form $[\alpha_j,\beta_j]$ with $1 \leq j \leq p$ \emph{bands} of $\sigma(J_{a,b})$, while intervals of the form $(\beta_j,\alpha_{j+1})$ with $1 \leq j \leq p-1$ are called \emph{gaps} of the spectrum. Notice that the unbounded components of the resolvent set are not considered gaps. If $\beta_j = \alpha_{j+1}$, we say that the $j$th gap of $\sigma(J_{a,b})$ is \emph{closed}. To avoid repeating the phrase ``with all gaps open,'' we will say that $J_{a,b}$ is $p$-\emph{generic} if both $a$ and $b$ are $p$-periodic and $\sigma(J_{a,b})$ has precisely $p$ connected components. The \emph{band-interior} of the spectrum will be defined by
$$
\sigma_{\Int}(J_{a,b})
=
\bigcup_{j=1}^p (\alpha_j,\beta_j).
$$
Of course, this is different from the topological interior of $\sigma(J_{a,b})$ whenever the spectrum has closed gaps.

Since we fix a periodic off-diagonal sequence $a$, we will think of genericity of $J_{a,b}$ as a property of $b$; specifically, we will say that $b$ is a $(p,a)$-generic potential if $J_{a,b}$ is $p$-generic.
\end{defi}

The following gap-opening lemma is a straightforward modification of \cite[Claim~3.4]{avila}. We include the proof with cosmetic alterations to the Jacobi case for the convenience of the reader.

\begin{lemma} \label{l:gen:dense}
Suppose $b$ is a $(p,a)$-generic potential and $k\geq 2$. For each $t \in \R$, define a $kp$-periodic sequence $b_t$ by
$$
b_t(n)
=
\begin{cases}
b(n) & 1 \leq n \leq kp-1 \\
b(kp) + t & n = kp
\end{cases}
$$
Then $b_t$ is $(kp,a)$-generic for all but finitely many choices of $t \in \R$. In particular, for any $\delta > 0$, there exists a $(kp,a)$-generic potential $b'$ with $\|b-b'\|_\infty < \delta$, so the generic potentials are dense in the space of periodic potentials.
\end{lemma}

\begin{proof}
If $\sigma(J_{a,b_t})$ has a closed gap at energy $E_t \in \R$, then it follows from Theorem~\ref{t:floq} that the matrix $A_{E_t}^{(a,b_t)}(kp)$ must be equal to $\pm I$. In particular, examining the unperturbed transfer matrices, we have
\begin{align*}
A_{E_t}^{(a,b)}(kp)
& =
\frac{1}{a(kp)}
\begin{pmatrix}
E_t - b(kp)& -1 \\
a(kp)^2 & 0
\end{pmatrix}
\cdot
A_{E_t}^{(a,b)}(kp-1)\\
& =
\begin{pmatrix}
E_t -  b(kp) & -1 \\
a(kp)^2 & 0
\end{pmatrix}
\begin{pmatrix}
E_t -  b(kp) - t  & -1 \\
a(kp)^2 & 0
\end{pmatrix}^{-1}
\cdot
A_{E_t}^{(a,b_t)}(kp)\\
& = 
\pm 
\begin{pmatrix}
1 &  t \\
0 & 1
\end{pmatrix}.
\end{align*}
In particular, if $t \neq t'$, this forces $A_{E_t}^{(a,b)}(kp) \neq A_{E_{t'}}^{(a,b)}(kp)$ and hence $E_t \neq E_{t'}$. Moreover, if $t \neq 0$, then the above implies that
$$
A_{E_t}^{(a,b)}(p)
=
\begin{pmatrix} \pm 1 & \ast \\ 0 & \pm 1 \end{pmatrix},
$$ 
which implies that the discriminant of $J_{a,b}$ is $\pm 2$ at $E_t$. Since the discriminant of $J_{a,b}$ is a polynomial of degree $p$ in $E$, there can be at most $2p$ distinct values of $E$ for which it attains the values $\pm 2$, and hence the lemma follows.
\end{proof}

Let us set up a bit of terminology. The idea here is the following: we start with a $(p,a)$-generic potential $b$ and then perform a small perturbation of $b$ to produce a $(kp,a)$-generic potential $b'$. Of course, if the perturbation is small enough, then the spectrum of $J_{a,b'}$ will inherit $p-1$ gaps from the spectrum of $J_{a,b}$ and will produce $(k-1)p$ new gaps. We want to control the locations at which these new gaps form. From the point of view of logarithmic potential theory, it is natural to partition $\sigma(J_{a,b})$ into $kp$ subintervals, each of which has harmonic measure $\frac{1}{kp}$; compare \cite[Section~5.5]{simszego}. The following definition precisely describes the endpoints of these subintervals.

\begin{defi}
Let $J=J_{a,b}$ be a periodic Jacobi matrix with corresponding discriminant $D$ (recall that we have fixed the periodic background off-diagonal sequence $a$). We say that $E \in \R$ is a $k$-\emph{break point} of $J$ if it satsifies
\begin{equation} \label{eq:breakpt:def}
D(E)
=
2\cos\left(\frac{\pi j}{k}\right)
\,
\text{ for some integer } 0 \leq j \leq k.
\end{equation}
We say that $E$ is a \emph{proper} break point if $1 \leq j \leq k-1$; equivalently, the improper break points of $J$ are simply the edges of bands of the spectrum of $J$.
\end{defi}

By Theorem~\ref{t:floq}, a $p$-periodic Jacobi operator will have precisely $(k-1)p$ \emph{proper} $k$-break points. It is not hard to see that every solution of \eqref{eq:eveq} is $kp$-periodic whenever $E$ is a proper $k$-break point of $J$ with $j$ even; similarly, every solution of \eqref{eq:eveq} is $kp$-antiperiodic whenever $E$ is a proper $k$-break point of $J$ with odd $j$. In particular, the $k$-break points of $J$ are precisely the eigenvalues of $J$ restricted to $[1,kp]$ with periodic or antiperiodic boundary conditions. More precisely,  the set of $k$-break points of $J$ is precisely the set of eigenvalues of $J_{a,b}^{kp,\pm}$ (where $J_{a,b}^{kp,\pm}$ is defined as in \eqref{eq:bprob}). Moreover, from the discussion above, it is easy to see that any proper $k$-break point of $J$ is a doubly degenerate eigenvalue of one of $J_{a,b}^{kp,\pm}$. 

As the name suggests, when we perturb $b$ slightly to produce $b'$, then the $(k-1)p$ new small gaps form near the proper break points, provided $b'$ is sufficiently close to $b$. This is a relatively straightforward consequence of standard eigenvalue perturbation theory, since the band edges of $\sigma(J_{a,b'})$ are precisely the eigenvalues of $J_{a,b'}^{kp,\pm}$ by Theorem~\ref{t:floq}. The precise statement follows.

\begin{lemma} \label{l:breakpts}
Suppose $J = J_{a,b}$ is a $p$-generic Jacobi operator and $k \geq 2$. If $\varepsilon > 0$ is sufficiently small, $b'$ is $(kp,a)$-generic, and
$$
\|b-b'\|_\infty < \varepsilon,
$$
then, for each proper $k$-break point $E$ of $J$, there exists a gap of $\sigma(J_{a,b'})$ entirely contained within $B_\varepsilon(E)$. Each of the remaining $p - 1$ gaps of $\sigma(J_{a,b'})$ is contained in a $\varepsilon$-neighborhood of a gap of $\sigma(J)$. Indeed, this conclusion holds as soon as $\varepsilon$ is less than one-half the minimum distance between distinct $k$-break points of $J$.
\end{lemma}

\noindent \textit{Remark.} If one does not assume that $b'$ is $(kp,a)$-generic, the proof still provides useful information about the structure of $\sigma(J_{a,b'})$. Specifically, if $b'$ is simply $kp$-periodic and $\|b - b'\| < \delta$, then the proof shows that $b'$ has at least $p-1$ gaps, each of which is contained in an $\varepsilon$-neighborhood of a gap of $\sigma(J)$, while any other gaps of $\sigma(J_{a,b'})$ must form in $\varepsilon$-neighborhoods of proper break points, with at most one new gap in each $\varepsilon$-neighborhood of a proper break point of $J$.
\newline

\begin{proof}[Proof of Lemma~\ref{l:breakpts}]

Let $\varepsilon > 0$ be given as in the statement of the lemma. Notice that the condition on $\varepsilon$ means that the $\varepsilon$-neighborhoods of the $k$-break points (of $J$) are pairwise disjoint. Now, suppose that $b'$ is $(kp,a)$-generic and satisfies $\|b-b'\|_\infty <\varepsilon$. Applying Proposition~\ref{p:specdist} to $J_{a,b}^{kp,+}$ and $J_{a,b'}^{kp,+}$ (resp., to $J_{a,b}^{kp,-}$ and $J_{a,b'}^{kp,-}$), we see that each eigenvalue of $J_{a,b'}^{kp,\pm}$ must be within $\varepsilon$ of an eigenvalue of $J_{a,b}^{kp,\pm}$. Using this, disjointness of the $\varepsilon$ neighborhoods of eigenvalues of $J_{a,b}^{kp,\pm}$, and Theorem~\ref{t:floq}, the lemma follows easily.

\end{proof}

\section{Proofs of Theorems~\ref{t:jacobi:homog} and \ref{t:jacobi:cantgroup:homog}} \label{sec:proof}

\begin{proof}[Proof of Theorem~\ref{t:jacobi:homog}]
Fix a $(p_0,a)$-generic potential $b_0$ and constants $\tau < 1$, $\varepsilon > 0$. By Lemma~\ref{l:gen:dense}, the generic potentials are dense in $\mathcal L$, so, to prove the theorem, it suffices to construct an element of $\mathcal H_\tau^a$ in $B_\varepsilon(b_0)$. To that end,  let $\lambda_0$ be the minimal length of a band of $\sigma_0 := \sigma(J_{a,b_0})$, and let $\gamma_0$ be the minimal length of a gap of $\sigma_0$. Put $\varepsilon_0 = \min(\gamma_0, 4\varepsilon)$ and fix a sequence\footnote{For this theorem, it is not necessary to use an arbitrary sequence of $k$'s. However, we will need this freedom to prove Theorem~\ref{t:jacobi:cantgroup:homog}, since not all periodic potentials fiber over an arbitrary Cantor group.} of integers $k_1, \, k_2,\ldots\geq 2$. The $k$'s will control the periods of approximants viz.\ $p_n = k_n p_{n-1}$ for each $n \in \Z_+$. Now, let $t_0$ be the minimal distance between consecutive $k_1$-break points of $J_{a,b_0}$ and choose a sequence $r_1 > r_2 > \cdots  > 0$ so that
\begin{equation} \label{eq:r:conds}
\sum_{\ell=1}^\infty r_\ell < 1 - \tau.
\end{equation}

Now, we will inductively choose a sequence of positive numbers $(\varepsilon_j)_{j=1}^\infty$ and a sequence of potentials $(b_j)_{j=1}^\infty$ in such a way that the bands of $\sigma(J_{a,b_j})$ are very long relative to the spectral gaps which are introduced in passing from $b_{j-1}$ to $b_j$. First, for the sake of notation, denote by $\lambda_j$ the length of the shortest band of $\sigma_j = \sigma(J_{a,b_j})$, let $\gamma_j$  be the minimal gap length of $\sigma_j$, and denote by $t_j$ the minimal distance between consecutive $k_{j+1}$-break points of $J_{a,b_j}$. In this notation, we may choose the sequences $(\varepsilon_j)_{j=1}^\infty$ and $(b_j)_{j=1}^\infty$ so that the following properties hold:
\begin{itemize}
\item One has
\begin{equation}\label{eq:fastconv}
\varepsilon_j
<
\min\left( 
\frac{\gamma_{j-1}}{5}, 
\frac{\varepsilon_{j-1}}{5},
\frac{t_{j-1}}{2}, 
\frac{r_j t_{j-1}}{2r_j + 5},
e^{-j \cdot p_{j+1}} \right)
\end{equation}
for all $j \geq 1$.
\item For every $j \in \Z_+$, $b_j$ is $(p_j,a)$-generic.
\item For all $j$, $\|b_j - b_{j-1}\|_\infty < \varepsilon_j$.
\end{itemize}

We begin by noticing several consequences of these conditions. First, the condition $\varepsilon_j < t_{j-1} /2$ means that the conclusion of Lemma~\ref{l:breakpts} holds with $b = b_{j-1}$, $b' = b_j$, and $\varepsilon = \varepsilon_j$. More precisely, $p_{j-1} - 1$ gaps of $\sigma_j$ are contained in $\varepsilon_j$-neighborhoods of gaps of $\sigma_{j-1}$, each $\varepsilon_j$-neighborhood of a proper $k_j$-break point of $J_{a,b_{j-1}}$ contains exactly one gap of $\sigma_j$, and this is an exhaustive list of all gaps of $\sigma_j$.

As a consequence of the preceding paragraph, the assumptions on $b_j$, and Lemma~\ref{l:breakpts}, we get $\lambda_j \geq t_{j-1} - 2\varepsilon_j$ for each $j \geq 1$. In particular, using this and the fourth condition in \eqref{eq:fastconv}, we obtain
\begin{equation} \label{eq:fastconv2}
\varepsilon_j
\leq 
\frac{\varepsilon_j \lambda_j}{t_{j-1} - 2\varepsilon_j}
<
\frac{\lambda_j r_j}{5}
\end{equation}
for all $j \in \Z_+$.
\newline

To establish the desired spectral homogeneity, we will prove the estimate
\begin{equation} \label{eq:step:homog}
|B_\delta(E) \cap \sigma_n |
\geq
\delta\left( 1 - \sum_{\ell=1}^n r_\ell \right)
\text{ for all }
0 < \delta \leq \lambda_0, \,
E \in \sigma_n
\end{equation}
for all $n \in \Z_+$. Fix $n \in \Z_+$, $E \in \sigma_n$, and $0 < \delta \leq \lambda_0$. If $\delta \leq \lambda_{n}$, then $B_\delta(E)$ contains a subinterval of length $\delta$ which is completely contained in $\sigma_n$, which implies
$$
|B_\delta(E) \cap \sigma_n|
\geq
\delta.
$$
Next, assume that $\lambda_j < \delta \leq \lambda_{j - 1}$ for some $1 \leq j \leq n$. By Proposition~\ref{p:specdist}, there exists $E_0 \in \sigma_{j-1}$ with $|E - E_0| \leq \varepsilon_{j,n} := \varepsilon_j + \cdots + \varepsilon_{n}$. The key inequality in this step is
\begin{equation} \label{eq:key:ineq}
\varepsilon_{j,n} +  \sum_{\ell = j}^{n} 2 \varepsilon_\ell \left( \frac{\delta}{\lambda_\ell} +1 \right)
<
\delta \sum_{\ell = j}^{n}  r_\ell,
\end{equation}
which follows from \eqref{eq:fastconv2}, and $\delta > \lambda_j$ (hence $\delta > \lambda_\ell$ for every $\ell \geq j$). It is easy to see that there exists an interval $I_0$ of length $\delta - \varepsilon_{j,n}$ which contains $E_0$ such that $I_0 \subseteq \sigma_{j-1} \cap B_\delta(E)$. Consequently, we have the following estimates:
\begin{align*}
|B_\delta(E) \cap \sigma_{n}|
& \geq 
\left|I_0 \cap \sigma_{n} \right| \\
& \geq
\left|I_0  \cap \sigma_{j-1} \right|
-
\sum_{\ell = j}^{n} \left| I_0  \cap (\sigma_{\ell - 1} \setminus \sigma_\ell) \right| \\
& = 
(\delta - \varepsilon_{j,n}) - \sum_{\ell = j}^{n} \left| I_0  \cap (\sigma_{\ell - 1} \setminus \sigma_\ell) \right|.
\end{align*}
The third line uses $I_0 \subseteq \sigma_{j-1}$. Obviously, $I_0$ completely contains fewer than $\delta/\lambda_\ell$ bands of $\sigma_\ell$ for each $j \leq \ell \leq n$, so, by Proposition~\ref{p:specdist} and Lemma~\ref{l:breakpts}, we have
\begin{align*}
(\delta - \varepsilon_{j,n}) - \sum_{\ell = j}^{n} \left| I_0  \cap (\sigma_{\ell - 1} \setminus \sigma_\ell) \right|
& \geq
(\delta - \varepsilon_{j,n})
-
 \sum_{\ell = j}^{n} 2 \varepsilon_\ell \left( \frac{\delta}{\lambda_\ell} + 1 \right) \\
& >
\delta\left(1 - \sum_{\ell = j}^{n} r_\ell \right) \\
& \geq 
\delta\left(1 - \sum_{\ell = 1}^{n} r_\ell \right),
\end{align*}
where the penultimate line follows from \eqref{eq:key:ineq}. Thus, \eqref{eq:step:homog} holds for every $n \in \Z_+$. Consequently, by \eqref{eq:fastconv} and our choice of $\varepsilon_0$, we have a limiting potential $b_\infty := \lim b_n$ with
$$
\| b_0 - b_\infty \|_\infty 
<
\sum_{\ell=1}^\infty \varepsilon_\ell 
<
\varepsilon_0 \sum_{\ell = 1}^\infty 5^{-\ell}
=
\frac{\varepsilon_0}{4}
\leq
\varepsilon.
$$
Moreover, with $\sigma_\infty := \sigma(J_{a,b_\infty})$, we have
$$
|B_\delta(E) \cap \sigma_\infty|
\geq 
\delta\left(1 - \sum_{\ell=1}^\infty r_\ell\right)
>
\delta \tau
$$
for all $E \in \sigma_\infty$ and $0 < \delta \leq \lambda_0$ by \eqref{eq:step:homog} and Propositions~\ref{p:lebmsr:semicont} and \ref{p:specdist}. Thus, $\sigma_\infty$ is $\tau$-homogeneous.
\newline

To see that $\sigma_\infty$ is a Cantor set, it suffices to check that it is nowhere dense, since it cannot have isolated points by general principles \cite[Theorem~1]{pastur80}. To that end, let $U \subseteq \R$ be an open interval, and choose $n$ so that $4 \pi A^2/p_n < |U|$, where $A = \max(1, \|a\|_\infty)$. By Theorem~\ref{l.avila.lp1}, $U$ must contain an open subinterval $G$ of length $\gamma_n$ with $G \cap \sigma_n = \emptyset$, since $b_n$ is $(p_n,a)$-generic. Notice that \eqref{eq:fastconv} implies that
$$
\|J_{a,b_n} - J_{a,b_\infty} \|
=
\| b_n - b_\infty \|_\infty
<
\sum_{\ell=n+1}^\infty \varepsilon_\ell
<
\gamma_n \sum_{k=1}^\infty 5^{-k}
=
\frac{\gamma_n}{4}.
$$
Consequently, if $c$ denotes the center of $G$, then $c \notin \sigma_\infty$ by Proposition~\ref{p:specdist}. Thus, $\sigma_\infty$ is nowhere dense, as desired.
\newline

Finally, purely absolutely continuous spectrum is an immediate consequence of \eqref{eq:fastconv} and a discrete analog of the theorem of Pastur-Tkachenko due to Egorova \cite{egorova}. Specifically, \eqref{eq:fastconv} implies that
$$
\| J_{a,b_\infty} - J_{a,b_n} \|
\leq
e^{-n \cdot p_{n+1}} \cdot \sum_{j=1}^\infty 5^{-j}
<
e^{-n p_{n+1}}.
$$
Consequently, one obtains
$$
\lim_{n \to \infty} e^{\tilde C p_{n+1}} \| J_{a,b_\infty} - J_{a,b_n} \|
=
0
$$
for every $\tilde C > 0$, which implies that $J_{a,b_\infty}$ has purely absolutely continuous spectrum  by \cite{egorova}.
\end{proof}

To prove Theorem~\ref{t:jacobi:cantgroup:homog}, we need to introduce some more machinery. For the remainder of the section, assume that $\Omega$ is a fixed monothetic Cantor group with topological generator $\theta$. We say that $f \in C(\Omega,\R)$ is a $p$-periodic sampling function if $f\circ T^p = f$, where $T:x \mapsto x+\theta$. This is obviously equivalent to the statement that $s_\omega^f$ is a $p$-periodic sequence for every $\omega \in \Omega$, where $s_\omega^f$ is defined as in \eqref{eq:lp:cantdef}. Since $\Omega$ is profinite and monothetic, there exists a sequence $\Omega_1 \supseteq \Omega_2 \supseteq \cdots$ of compact finite-index subgroups of $\Omega$ with the property that
$$
\bigcap_{j=1}^\infty \Omega_j
=
\{0\}.
$$
Let $n_j$ denote the index of $\Omega_j$ in $\Omega$. The following proposition is not hard to prove; compare \cite[Section~3]{avila}.

\begin{prop} \label{p:per:sf}
Let $f \in C(\Omega,\R)$. Then $f$ is an $n_j$-periodic sampling function if and only if it descends to a well-defined function on the quotient $\Omega/\Omega_j$. Moreover, any periodic sampling function is defined over some quotient of the form $\Omega/\Omega_j$ with $j \geq 1$. Consequently, if $b$ is a periodic sequence with period $p$ which divides $n_j$ for some $j$, then $b = b_0^f$ for suitable $f \in C(\Omega,\R)$, where $b_0^f(n) = f(n \theta)$, as usual.
\end{prop}

\begin{proof}[Proof of Theorem~\ref{t:jacobi:cantgroup:homog}]

Suppose $\varepsilon > 0$ and $f$ is an $(n_q,a)$-generic sampling function for some $q \geq 1$. As before, the generic sampling functions are dense, so it suffices to find an element of $\mathcal H_\tau^a$ in $B_\varepsilon(f)$. Let $b_0 = s_0^f$ as in \eqref{eq:lp:cantdef}. Define $k_j = n_{q+j}/n_{q+j-1}$ so that $p_j = n_{q+j}$, and choose $p_j$-generic potentials $b_j$ exactly as in the proof of Theorem~\ref{t:jacobi:homog}. In particular, $b_\infty = \lim b_j$ is such that $J_{a,b_\infty}$ has all of the desired properties. By Proposition~\ref{p:per:sf}, there exist $f_j \in C(\Omega,\R)$ such that $b_j = b_0^{f_j}$ for each $j$. It is not hard to see that $f_\infty = \lim f_j$ exists and that $b_0^{f_\infty} = b_\infty$, so the theorem is proved.
\end{proof}

\section{The CMV Case}\label{sec:cmvproofs}

In this section, we discuss the modifications to the proofs of Theorems~\ref{t:jacobi:homog} and \ref{t:jacobi:cantgroup:homog} necessary to obtain Theorems~\ref{t:homog:cmv} and \ref{t:homog:cmv:cantgroup}. In essence, no extra work is needed -- one simply needs to find suitable replacements for the various pieces which comprise the proofs and then re-run the entire machine. 

First, we replace Lebesgue measure on $\R$ with arc-length measure on $\partial \D$, that is, the pushforward of Lebesgue measure on $[0,2\pi)$ under the map $t \mapsto \exp(it)$. Equivalently, arc-length measure on $\partial \D$ can simply be thought of as one-dimensional Hausdorff measure. Clearly, there is a version of the Hausdorff metric for compact subsets of $\partial \D$, also defined by the formula \eqref{eq:hdmetric:def}. Here, the $\varepsilon$-neighborhoods of sets should of course be thought of as $\varepsilon$-neighborhoods with respect to the usual metric on $\C$. It is then trivial to modify Propositions~\ref{p:lebmsr:semicont} and \ref{p:specdist} to fit this setting. The precise statements follow.

\begin{prop}
If $(A_n)_{n=1}^\infty$ and $(B_n)_{n=1}^\infty$ are sequences of compact subsets of $\partial \D$ such that $A_n \to A$ and $B_n \to B$ with respect to the Hausdorff metric, then
$$
|A \cap B| \geq \limsup_{n \to \infty} |A_n \cap B_n|,
$$
where $| \cdot |$ denotes arc-length measure on $\partial \D$.
\end{prop}

\begin{prop} \label{p:unit:specdist}
If $U$ and $V$ are unitary operators, then
$$
d_{\Hd}(\sigma(U),\sigma(V))
\leq
\|U - V\|.
$$
\end{prop}

One also has a version of Floquet theory for periodic CMV matrices; that is, if $\alpha \in \D^{\Z}$ is $p$-periodic, then the spectrum of $\E = \E_\alpha$ consists of $p$ nonoverlapping closed subarcs of $\partial \D$, which can be found by examining a degree $p$-polynomial $D$, just as in the Jacobi case. As before, we say that $\E$ is $p$-generic if $\sigma(\E)$ consists of precisely $p$ connected components. In this setting, it is known that the $p$-generic CMV operators are dense in the space of all $p$-periodic CMV operators \cite[Theorem~11.13.1]{simopuc2}. There is a slight combinatorial difference here, namely, that $p$-generic CMV matrices have $p$ spectral gaps, not $p-1$.

The analog of the band-length estimate in Theorem~\ref{l.avila.lp1} is proved in \cite[Lemma~5]{ong12}. Specifically, if $\alpha$ is $p$-periodic, then
\begin{equation} \label{eq:cmv:bandmeas}
|A|
\leq
\frac{2\pi}{p}
\end{equation}
for each band $A \subseteq \sigma(\E_\alpha)$. Using Floquet theory for periodic CMV matrices, we can define $k$-break points of $\E$ in exactly the same way, namely, by partitioning each band of the spectrum into $k$ closed subarcs, each of which has harmonic measure $\frac{1}{kp}$. One can then prove a straightforward modification of Lemma~\ref{l:breakpts}.

\begin{lemma} \label{l:cmv:breakpts}
Suppose $\E = \E_\alpha$ is a $p$-generic CMV matrix and $k \geq 2$. For all $\varepsilon > 0$ sufficiently small, there exists $\delta > 0$ such that if $\alpha'$ is $kp$-generic and
$$
\|\alpha - \alpha'\| < \delta,
$$
then, for each proper $k$-break point $z$ of $\E$, there exists a gap of $\sigma(\E_{\alpha'})$ entirely contained within $B_\varepsilon(z)$. Each of the remaining $p$ gaps of $\sigma(\E)$ is contained in an $\varepsilon$-neighborhood of some gap of $\sigma(\E)$.
\end{lemma}

The proof is essentially the same as before, with  mostly cosmetic variations on the main theme. There is one minor annoyance in this case. Specifically, in the Jacobi case, we (implicitly) used the obvious identity
$$
\|J_{a,b} - J_{a,b'} \|
=
\| b - b' \|_\infty
$$
when we invoked Proposition~\ref{p:specdist}, and this does not translate directly to the CMV context. Instead, one has 
\begin{equation} \label{eq:cmv:holder}
\|\E_\alpha - \E_{\alpha'}\|^2
\leq
72 \|\alpha - \alpha'\|_\infty,
\end{equation}
by \cite[(4.3.11)]{simopuc1}. This simply introduces some constants which have no qualitative impact on the structure of the proof. With this variant of Lemma~\ref{l:breakpts} in hand, the proofs from Section~\ref{sec:proof} can be rerun with minor changes.

\begin{appendix}

\section{A Band Length Estimate for Periodic Jacobi Operators}

In this appendix, we provide a proof of a band length estimate for periodic Jacobi operators which is analogous to \eqref{eq:cmv:bandmeas} and \cite[Lemma~2.4(1)]{avila}. Specifically, we have the following upper bound.

\begin{theorem}\label{l.avila.lp1}
Suppose $J$ is a $p$-periodic Jacobi matrix, and let 
$$
A
=
\max(1,a_1,\ldots,a_p).
$$
The Lebesgue measure of any band of $\sigma(J)$ is bounded above by $\frac{2 \pi A^2}{p}$.
\end{theorem}

In order to prove the desired band length estimate, we need to discuss the integrated density of states for periodic Jacobi operators. In particular, we will elucidate a point of view on the IDS of periodic operators discussed in \cite{AD08}. This is a special case of general, powerful formulas for absolutely continuous spectrum; see \cite{deiftsimon83}. The material in this appendix is standard and well-known within the community, but we opted to present it here, since \cite{avila} and \cite{AD08} do not work out the details explicitly, and these references work exclusively in the discrete Schr\"odinger setting, where $a \equiv 1$ (except \cite{deiftsimon83}, which also works out a similar framework for continuum Schr\"odinger operators).

In general, for a Jacobi matrix $J$, the corresponding integrated density of states, $k$, is defined by the limit
\begin{equation}\label{eq:ids:def}
k(E)
=
\lim_{N \to \infty} \frac{1}{2N+1} \#\{ \lambda \in \sigma(J_N) : \lambda \leq E \},
\end{equation}
whenever the limit exists, where $J_N$ denotes the restriction of $J$ to the interval $[-N,N]$ with Dirichlet boundary conditions, i.e.
$$
J_N =
\begin{pmatrix}
b(-N) & a(-N)    &        &        &  \\
a(-N)    & b(-N+1) & a(-N+1)      &        &   & \\
     & \ddots    & \ddots & \ddots &   & \\
          &     & a(N-2) & b(N-1) & a(N-1)  & \\
   &      &        & a(N-1)      & b(N)
\end{pmatrix}.
$$
It is a well-known fact that the the limit on the right-hand side of \eqref{eq:ids:def} exists whenever $J$ is a $p$-periodic Jacobi matrix. Using  Floquet theory, one can explicitly describe $k$ in terms of the discriminant, $D$. The following theorem is standard; see \cite[Theorem~5.4.5]{simszego}.

\begin{theorem} \label{t.periodic.ids}
Suppose $J$ is a $p$-periodic Jacobi matrix, with corresponding discriminant $D$ and integrated density of states $k$. Then $k$ is differentiable on on the interior of the spectrum. We have
\begin{equation}  \label{e.periodic.ids}
\frac{dk}{dE} 
= 
\frac{1}{\pi p} \left| \frac{d\theta}{dE} \right|,
\end{equation}
where $\theta = \theta(E)$ is chosen continuously so that
\begin{equation}\label{eq:pertheta:def}
2 \cos(\theta) = D(E)
\end{equation}
for $E \in \sigma_{\Int}(H)$. In particular, if $B$ is any band of the spectrum,
\begin{equation} \label{eq:per:bandmeas}
\int_B \! dk(E) = \frac{1}{p}.
\end{equation}
\end{theorem}

Recall that any $A \in \SL(2,\R)$ induces a linear fractional transformation on the upper half-plane $\C_+ = \{ z \in \C : \mathrm{Im}(z) > 0\}$ via
$$
\begin{pmatrix}
a & b \\
c & d
\end{pmatrix} \cdot z
= \frac{az + b}{cz+d}.
$$
For $\theta \in \R$, define
$$
R_\theta
=
\begin{pmatrix} \cos(\theta) & - \sin(\theta) \\ \sin(\theta) & \cos(\theta) \end{pmatrix}
$$
Obviously,
 \begin{equation}\label{eq:char:so2}
\SO(2)
=
\{ R_\theta : 0 \leq \theta < 2 \pi \}
=
\{ R \in \SL(2,\R) : R \cdot i = i \}.
\end{equation}

\begin{lemma} \label{l:elliptic:sl2}
A matrix $A \in \SL(2,\R)$ satisfies $|\tr(A)| < 2$ if and only if its action on $\C_+$ has a unique fixed point. Whenever $|\tr(A)| < 2$, there exists $M \in \SL(2,\R)$ such that
$$
MAM^{-1} = R_\theta \in \SL(2,\R),
$$ 
where $2 \cos(\theta) = \tr(A)$.  Moreover, such a conjugacy is unique modulo left-multiplication by an element of $\SO(2)$.
\end{lemma}

\begin{proof}
This is a consequence of straightforward calculations. 
\end{proof}

Now, for $E$ in the interior of a band, $ |D(E)| < 2$, so the monodromy matrix is conjugate to $R_\theta$, where $\theta$ satisfies $2 \cos(\theta) = D(E)$. From Theorem~\ref{t.periodic.ids}, we know that the derivative of the integrated density of states can be related to $ |d\theta / dE| $, so we would like to find some other way to recover this derivative.  By way of motivation, suppose $\theta$ is a smooth function of $t$.  It is then easy to check that
$$
R_\theta^{-1} \frac{dR_\theta}{dt} 
= 
\begin{pmatrix} 0 & - d\theta/dt \\ d\theta/dt & 0 \end{pmatrix}.
$$
This motivates us to define the \emph{anti-trace} of a $2 \times 2$ matrix by
$$
\atr \begin{pmatrix} a & b \\ c & d \end{pmatrix} = c - b.
$$
Like the usual trace, the anti-trace is a linear functional in the sense that 
$$
\atr(A + \lambda B) = \atr(A) + \lambda \, \atr(B)
$$
for all $\lambda \in \R$ and all  $A, B \in \R^{2\times 2}$.  However, unlike the trace, the anti-trace is not cyclic, i.e., one can have $\atr(AB) \neq \atr(BA)$.  Despite this, we still have the following weakened variant of cyclicity.

\begin{lemma} \label{l:atr:conj:inv}
If $R \in \SO(2)$ and $A \in \R^{2 \times 2}$,
\begin{equation} \label{eq:atr:conj:inv}
\atr\left(R^{-1} A R \right)
=
\atr(A).
\end{equation}
\end{lemma}

\begin{proof}
This is an easy calculation.
\end{proof}

\begin{lemma} \label{l:smoothconj:to:rot}
Suppose $I$ is an open interval and $\Phi: I \to \SL(2,\R)$ is a smooth map such that $\big| \tr(\Phi(t)) \big| < 2$ for all $t \in I$. Under these conditions, there exists a smooth choice of $M \in \SL(2,\R)$ such that
\begin{equation} \label{eq:smoothconj:to:rot}
M \Phi M^{-1} = R_\theta,
\end{equation}
where $2 \cos(\theta) = \tr(\Phi)$. Moreover, the angle $\theta$ can be chosen to be a smooth function of $t$; in this case, it satisfies
$$
\frac{d\theta}{dt} = \frac{1}{2} \atr \left( M \Phi^{-1} \frac{d\Phi}{dt} M^{-1} \right), 
$$
\end{lemma}

\begin{proof}
To construct the conjugacy $M$, first notice that the unique fixed point $z=z(t) \in \C_+$ of $\Phi$ varies smoothly with $t$.  We then define
$$
M(t)
=
\big( \mathrm{Im}(z(t)) \big)^{-1/2} \begin{pmatrix} 1 & - \mathrm{Re}(z(t)) \\ 0 & \mathrm{Im}(z(t)) \end{pmatrix},
$$
Evidently, the linear fractional transformation corresponding to $M \Phi M^{-1}$ fixes $i$, which implies $M \Phi M^{-1} \in \SO(2)$. By cyclicity of the trace, $M \Phi M^{-1}$ must be of the claimed form. Differentiating the relation \eqref{eq:smoothconj:to:rot} using the product rule, one obtains
$$
\frac{dM}{dt} \Phi M^{-1} + M \frac{d\Phi}{dt} M^{-1} + M \Phi \frac{dM^{-1}}{dt}   = \frac{dR}{dt}
=
R \begin{pmatrix} 0 & - d\theta/dt \\ d\theta/dt & 0 \end{pmatrix}.
$$
Multiply on the left by $R^{-1}$ and simplify using \eqref{eq:smoothconj:to:rot} to obtain
\begin{equation} \label{eq:smoothconj:to:rot:step}
R^{-1} \frac{dM}{dt} M^{-1} R + M \Phi^{-1} \frac{d\Phi}{dt} M^{-1} + M\frac{d M^{-1}}{dt}
=
\begin{pmatrix} 0 & - d\theta/dt \\ d\theta/dt & 0 \end{pmatrix}.
\end{equation}
By \eqref{eq:atr:conj:inv}, linearity of the anti-trace, and the product rule,
\begin{align*}
\atr \left(R^{-1} \frac{dM}{dt} M^{-1} R + M \frac{dM^{-1}}{dt}\right)
& =
\atr\left(  \frac{dM}{dt} M^{-1} + M \frac{d M^{-1}}{dt} \right) \\
& =
\atr\left( \frac{d}{dt}(MM^{-1}) \right) \\
& =
0.
\end{align*}
Thus, \eqref{eq:smoothconj:to:rot} follows by taking the anti-trace of \eqref{eq:smoothconj:to:rot:step}.
\end{proof}

We can use the preceding lemma to find another way to view the integrated density of states of a periodic Jacobi operator via Hilbert-Schmidt norms of conjugacies between monodromy matrices and rotations. Specifically, suppose $J$ is $p$-periodic and denote 
$$
T_j = \frac{1}{a(j)} \begin{pmatrix} E - b(j) & - 1 \\ a(j)^2 & 0 \end{pmatrix}, 
\quad
A_j = T_{j} \cdots T_1,
\quad
\Phi_j = A_{j-1} A_p A_{j-1}^{-1},
\quad
\text{for } j \geq 1,
$$
where we adopt the convention $A_0 = I$ in the $j = 1$ case of the final definition and suppress the dependence of all quantities on $E$ for notational simplicity. For $E \in \sigma_{\Int}(H)$, choose $M_j \in \SL(2,\R)$ such that $M_j \Phi_j M_j^{-1} \in \SO(2)$.

\begin{theorem} \label{t:per:ids:hilbschmidt}
Let $J$ be a $p$-periodic Jacobi operator with corresponding integrated density of states $k$, and put 
$$
A = \max(a_1,\ldots,a_p,1).
$$
We have
$$
\frac{dk}{dE}
\geq
\frac{1}{4 \pi A^2 p} \sum_{j=1}^p \| M_j \|_2^2
$$
on $\sigma_{\Int}(H)$, where $ \| M \|_2 = \sqrt{\tr(M^* M)} $ denotes the Hilbert-Schmidt norm of $M$.
\end{theorem}

\begin{proof}
First, notice that $\|M_j\|_2$ does not depend on the choice of conjugacy, for any other conjugacy from $\Phi_j$ to a rotation must take the form $O M_j$ for some $O \in \SO(2)$ by Lemma~\ref{l:elliptic:sl2}.  Since we may take $M_j$ to be given by the explicit formula
$$
M_j
=
\big( \mathrm{Im}(z_j) \big)^{-1/2}
\begin{pmatrix}
1 & - \mathrm{Re}(z_j) \\
0 & \mathrm{Im}(z_j)
\end{pmatrix}
$$
we see that
$$
\| M_j \|_2^2 
=
\frac{1 + |z_j|^2}{\mathrm{Im}(z_j)},
$$
where $z_j$ is the unique fixed point of the action of $\Phi_j$ on $\C_+$. Notice that $T_j z_j = z_{j+1}$ and hence $M_{j+1} T_j M_j^{-1} $ fixes $i$, so $M_{j+1} T_j M_j^{-1} =: Q_j \in \SO(2)$. One can easily compute
$$ 
T_j^{-1} \frac{dT_j}{dE} = \begin{pmatrix} 0 & 0 \\ -1 & 0 \end{pmatrix}, 
$$ 
Thus, by the product rule, we have
$$
\Phi_1^{-1} \frac{d\Phi_1}{dE}
=
\sum_{j=1}^{p} A_{j-1}^{-1} \begin{pmatrix} 0 & 0 \\ -1 & 0 \end{pmatrix} A_{j-1}
$$
With $R_j = Q_{j} \cdots Q_1$ and $R_0 = I$, we have
\begin{equation} \label{eq:per:ids:hilbschmidt:deriv}
\Phi_1^{-1} \frac{d\Phi_1}{dE}
=
\sum_{j=1}^{p} M_1^{-1} R_{j-1}^{-1} M_{j} \begin{pmatrix} 0 & 0 \\ -1 & 0 \end{pmatrix} M_{j}^{-1} R_{j-1} M_1
\end{equation}
To find the rate of change of $\theta$ with respect to $E$, we apply Lemma~\ref{l:smoothconj:to:rot} and compute
\begin{align*}
 \frac{d\theta}{dE}
& = 
\frac{1}{2} \atr \left( M_1 \Phi_1^{-1} \frac{d\Phi_1}{dE} M_1^{-1} \right)  \\
& =
\frac{1}{2} \atr\left( \sum_{j=1}^{p} M_{j} \begin{pmatrix} 0 & 0 \\ -1 & 0 \end{pmatrix} M_{j}^{-1} \right)  \\
& =
\frac{1}{2} \sum_{j=1}^{p} \frac{|z_j|^2}{\mathrm{Im}(z_j)}.
\end{align*}
The second line follows from \eqref{eq:per:ids:hilbschmidt:deriv} and Lemma~\ref{l:atr:conj:inv}, and the final line is a straightforward computation from the explicit form of $M_j$. An easy calculation using $ T_j z_j = z_{j+1} $ reveals
$$
\mathrm{Im}(z_{j+1})
=
\mathrm{Im}\left( \frac{E - b_{j}}{a_j^2} - \frac{1}{a_j^2 z_j} \right)
=
\frac{\mathrm{Im}(z_j)}{a_j^2|z_j|^2},
$$
which implies
\begin{align*}
\left| \frac{d\theta}{dE} \right|
& =
\frac{1}{2}\sum_{j=1}^{p} \frac{|z_j|^2}{\mathrm{Im}(z_j)} \\
& =
\frac{1}{4} \sum_{j=1}^p \left( \frac{1}{a_{j-1}^2 \mathrm{Im}(z_j)} + \frac{|z_j|^2}{\mathrm{Im}(z_j)} \right) \\
& \geq
\frac{1}{4A^2} \sum_{j=1}^{p} \frac{1+|z_j|^2}{\mathrm{Im}(z_j)} \\
& =
\frac{1}{4A^2} \sum_{j=1}^{p} \|M_j\|_2^2.
\end{align*}
Thus, the conclusion of the Theorem follows from Theorem~\ref{t.periodic.ids}.
\end{proof}

With this fact in hand, the desired estimate on the bands is easy.

\begin{proof}[Proof of Theorem~\ref{l.avila.lp1}]
Let $B$ denote a band of $\sigma(J)$. Using \eqref{eq:per:bandmeas} and Theorem~\ref{t:per:ids:hilbschmidt}, one has
$$
\frac{1}{p}
=
\int_{B} \! dk(E)
\geq
\int_{B} \left( \frac{1}{4 \pi A^2 p} \sum_{j=1}^p \|M_j\|_2^2 \right) dE
\geq
\frac{|B|}{2 \pi A^2},
$$
where we have used the bound $\|M\|^2_2 \geq 2$ which holds for any $M \in \SL(2,\R)$ (by Cauchy-Schwarz) in the final inequality. The theorem follows.
\end{proof}

\end{appendix}

\end{document}